\documentclass[final,a4paper,11pt]{article}
\usepackage{graphicx}
\usepackage{amssymb,amsmath, amsfonts,amsthm}
\usepackage{a4wide}
\usepackage{color}

\newcommand{\green}[1]{\textcolor{green}{#1}}
\newcommand{\R}{\mathbb{R}}

\newcommand{\norm}[1]{\|#1\|}

\newcommand{\dist}[1]{{\rm dist}(#1)}
\newcommand{\distb}[1]{{\rm dist}\big(#1\big)}
\newcommand{\mv}{\,\mid\,}

\newcommand{\K}{{\cal K}}

\newcommand{\F}{{\cal F}}
\newcommand{\Lag}{{\cal L}}

\newcommand{\xb}{\bar x}

\newcommand{\zb}{\bar z}

\newcommand{\pb}{\bar p}
\newcommand{\lb}{{\bar\lambda}}

\newcommand{\oo}{o}
\newcommand{\OO}{{O}}

\newcommand{\lin}{{\rm lin\,}}

\newcommand{\gph}{\mathrm{gph}\,}

\newcommand{\Gr}{{\rm gph\,}}
\newcommand{\Limsup}{\mathop{{\rm Lim}\,{\rm sup}}}
\def\h{\hfill\Box}

\def\dom{\mbox{\rm dom}\,}

\def\tilde{\widetilde}

\newlength{\myparboxwidth}
\if{
\newtheorem{theorem}{Theorem}[section]
\newtheorem{proposition}[theorem]{Proposition}
\newtheorem{remark}[theorem]{Remark}
\newtheorem{lemma}[theorem]{Lemma}
\newtheorem{corollary}[theorem]{Corollary}
\newtheorem{definition}[theorem]{Definition}
\newtheorem{example}[theorem]{Example}
\newtheorem{assumption}[theorem]{Assumption}
}\fi

\newtheorem{definition}{Definition}
\newtheorem{theorem}{Theorem}
\newtheorem{lemma}{Lemma}
\newtheorem{example}{Example}
\newtheorem{corollary}{Corollary}
\newtheorem{proposition}{Proposition}
\newtheorem{remark}{Remark}

\title{On the Aubin property of solution maps to  parameterized  variational systems with implicit constraints}
\author{Helmut Gfrerer\thanks{Institute of Computational Mathematics, Johannes Kepler University Linz,
              A-4040 Linz, Austria, Email: helmut.gfrerer@jku.at              }
\and  Ji\v{r}\'{i} V. Outrata\thanks{Institute of Information Theory and Automation, Academy of Sciences of the Czech Republic, 18208 Prague, Czech Republic, and Centre for
              Informatics and Applied Optimization, Federation University of Australia, POB 663, Ballarat,  Vic 3350, Australia,  Email: outrata@utia.cas.cz
}}
\date{}
\begin{document}
\maketitle



{\bf Abstract.} In the paper a new sufficient condition for the Aubin property to a class of parameterized variational systems is derived. In these systems the constraints depend both on the parameter as well as on the decision variable itself and they include, e.g., parameter-dependent quasi-variational inequalities and implicit complementarity problems. The result is based on a general condition ensuring the Aubin property of implicitly defined multifunctions which employs the recently introduced notion of the directional limiting coderivative. Our final condition can be verified, however, without an explicit computation of these coderivatives. The procedure is illustrated by an example.


\section{Introduction}
The Aubin (Lipschitz-like) property is probably the most important extension of the Lipschitz continuity to multifunctions. It has been introduced by J.-P. Aubin in \cite{Au} (under a different name) and since that time it is widely used in variational analysis and its applications.
In \cite{GO3}, a new condition has been derived ensuring the Aubin property of implicitly defined multifunctions around a given reference point. To be precise, it concerns the multifunction $S: \mathbb{R}^{m} \rightrightarrows  \mathbb{R}^{n}$ defined by
\begin{equation}\label{eq-200}
S(p):= \{x \in  \mathbb{R}^{n} | 0 \in M(p,x)\},
\end{equation}
where $p$ is the {\em parameter}, $x$ is the {\em decision variable} and multifunction $M: \mathbb{R}^{m} \times \mathbb{R}^{n} \rightrightarrows  \mathbb{R}^{l}$ is given. In the form (\ref{eq-200}) we can write down a large class of parameterized optimization and equilibrium problems and so this condition can well be used, e.g., in {\em post-optimal analysis} (where $p$ corresponds to uncertain problem data) or in  problems with the so-called equilibrium constraints (where $p$ represents the control variable).

The application of this condition requires, however, the computation of the graphical derivative and the directional limiting coderivative of $M$, which may be quite demanding, e.g., in case of solution maps to variational systems. In \cite{GO4} the authors investigated from this point of view a class of variational systems, in which the (rather general) constraints did not depend on the parameter $p$. In this paper we intend to make a further step and consider a variational system, where the constraint set depends both on the parameter as well as on the decision variable itself. This generality permits to analyze the Aubin property of, among other things, rather complicated parameterized {\em quasivariational inequalities}  (QVIs). The used model comes from \cite{QVI} where, to its analysis, the authors employed some advanced tools of the generalized differential calculus of B. Mordukhovich \cite{JMAA, M1}. Among the results of \cite{QVI} one finds also a sufficient condition for the Aubin property of the associated solution map and our main aim here is  the sharpening of that condition on the basis of the results from \cite{GO3}. At the same time, despite of the increased complexity of the considered model, our final condition (Theorem \ref{Thm.6})  seems to be more workable than its counterpart in \cite[Theorem 5]{GO4}.

The paper is organized as follows. In the preliminary Section 2 we recall the used notions from variational analysis and formulate properly the considered problem. Sections 3 and 4 deal with the computation of the graphical derivative and the directional limiting coderivative of the investigated multifunction $M$, respectively. The results presented in Section 4 depend heavily on selected results from \cite{BGO}, where a rich calculus for directional limiting notions (normal cones, subdifferentials and coderivatives) has been  developed. The main results of the paper are then collected in Section 5. They include both the new criterion (sufficient condition) for the Aubin property of the solution map to the investigated variational system as well as a formula for the graphical derivative of this solution map which may be used, e.g., in some sensitivity issues. The usage of the suggested technique is illustrated by an academic example.

The following notation is employed. Given a set $A \subset \mathbb{R}^{n}, \ {\rm sp\,} A$ stands for the linear hull of $A$ and $A^{\circ}$ is the (negative) polar of $A$. For a convex cone $K,~ \lin~ K$ denotes the {\em lineality} space of $K$, i.e., the set $K - K$.  Further, $\mathbb{B},\mathbb{S}$  is the unit ball and the unit sphere, respectively. Given a vector $a \in \mathbb{R}^{n}$, $[a]$ is the linear space generated by $a$ and $[a]^{\perp}$ stands for the orthogonal complement to $[a]$. Finally, $\stackrel{A}{\rightarrow}$ means the convergence within a set $A$ and $ \Limsup$ stands for the Painlev\'{e}-Kuratowski set limit.

\section{Problem formulation and preliminaries}

In the first part of this section we introduce some notions from variational analysis which will be
extensively used throughout the whole paper. Consider first a general closed-graph multifunction
$F:\mathbb{R}^{n} \rightrightarrows \mathbb{R}^{z}$ and its inverse $F^{-1}:\mathbb{R}^{z}
\rightrightarrows \mathbb{R}^{n}$ and assume that $(\bar{u},\bar{v})\in \Gr F$.
\begin{definition} 
 We say that $F$ has the {\em Aubin property} around
$(\bar{u},\bar{v})$, provided there are neighborhoods $U$ of $\bar{u}$, $V$ of $\bar{v}$ and a constant
$\kappa > 0$ such that
\[
F(u_{1}) \cap V \subset F(u_{2})+\kappa \| u_{1}-u_{2} \| \mathbb{B} \mbox{ for all } u_{1}, u_{2} \in
U.
\]
$F$ is said to be {\em calm} at $(\bar{u},\bar{v})$, provided there is a neighborhood $V$ of
$\bar{v}$ and a constant $\kappa > 0$ such that
\[
F(u) \cap V \subset F(\bar{u})+\kappa \| u-\bar{u} \| \mathbb{B} \mbox{ for all } u \in
\mathbb{R}^{n}.
\]
\end{definition}

It is clear that the calmness is substantially weaker (less restrictive) than the Aubin property.
Furthermore, it is known that $F$ is calm at $(\bar{u},\bar{v})$ if and only if $F^{-1}$ is {\em
metrically subregular at} $(\bar{v},\bar{u})$, i.e., there is a neighborhood $V$ of $\bar{v}$ and a
constant $\kappa > 0$ such that
\begin{equation}\label{eq-300}
\dist{v,F(\bar{u})}\leq \kappa \dist{\bar{u},F^{-1}(v)} \mbox{ for all } v \in V,
\end{equation}
cf. \cite[Exercise 3H.4]{DoRo}.

To obtain directional versions of the above properties, consider a direction $d \in \mathbb{R}^{z}$, positive numbers $\varrho, \delta$ and define the set
\[
\mathcal{V}_{\varrho, \delta}(d):=\left\{v \in \varrho \mathbb{B} \left|~ \big\| \norm{d} v  - \norm{v} d \big\| \leq \delta\norm{v}\norm{d}\right.\right\}.
\]
We say that a set $\mathcal{V}$ is a {\em directional neighborhood} of  $d$ if there exist $\varrho, \delta > 0$ such that  $\mathcal{V}_{\varrho, \delta}(d)\subset \mathcal{V}$. Now, when the neighborhood $V$ in (\ref{eq-300}) is replaced by the set $\bar{v}+ \mathcal{V}$, we say that $F^{-1}$ is metrically subregular at $(\bar{u},\bar{v})$ {\em in direction} $d$.

 To conduct a thorough analysis of the above stability notions one typically makes use of some basic
 notions of  generalized differentiation, whose definitions are presented below.
 \begin{definition}\label{DefVarGeom}
 Let $A$  be a closed set in $\mathbb{R}^{n}$ and $\bar{x} \in A$. Then
\begin{enumerate}
 \item [(i)]
 $T_{A}(\bar{x}):=\Limsup\limits_{t\searrow 0} \frac{A-\bar{x}}{t}$
 is the {\em tangent (contingent, Bouligand) cone} to $A$ at $\bar{x}$ and
 $ \hat{N}_{A}(\bar{x}):=(T_{A}(\bar{x}))^{\circ} $
 is the {\em regular (Fr\'{e}chet) normal cone} to $A$ at $\bar{x}$.
 \item [(ii)]
 $ N_{A}(\bar{x}):=\Limsup\limits_{\stackrel{A}{x \rightarrow \bar{x}}} \hat{N}_{A}(x)$
 is the {\em limiting (Mordukhovich) normal cone} to $A$ at $\bar{x}$ and, given a direction $d \in\mathbb{R}^{n}$,
$ N_{A}(\bar{x};d):= \Limsup\limits_{\stackrel{t\searrow 0}{d^{\prime}\rightarrow
 d}}\hat{N}_{A}(\bar{x}+ td^{\prime})$
 is the {\em directional limiting normal cone} to $A$ at $\bar{x}$ {\em in direction} $d$.
 \end{enumerate}
\end{definition}

If $A$ is convex, then $\hat{N}_{A}(\bar{x})= N_{A}(\bar{x})$ amounts to the classical normal cone in the sense of convex analysis and we will  write $N_{A}(\bar{x})$.
By the definition, the limiting normal cone coincides with the directional limiting normal cone in direction
$0$, i.e.,  $N_A(\bar{x})=N_A(\bar{x};0)$, and $N_A(\bar{x};d)=\emptyset$ whenever $d\not\in T_A(\bar{x})$.

In the sequel, we will also make an extensive use of the so-called critical cone.
In the setting of Definition \ref{DefVarGeom} with an  additionally given vector $d^{*}\in \mathbb{R}^{n}$, the cone
\[
\mathcal{K}_{A}(\bar{x},d^{*}):= T_{A}(\bar{x}) \cap [d^{*}]^{\perp}
\]
is called the {\em critical cone} to $A$ at $\bar{x}$ with respect to $d^{*}$.

The above listed cones enable us to describe the local behavior of set-valued maps via various
generalized derivatives. Consider again the multifunction $F$ and the point $(\bar{u},\bar{v})\in \Gr
F$.

\begin{definition}\label{DefGenDeriv}
\begin{enumerate}
\item [(i)]
The multifunction $D F(\bar{u},\bar{v}):\mathbb{R}^{n} \rightrightarrows\mathbb{R}^{z}$, defined by
\[
DF(\bar{u},\bar{v})(d):= \{h \in \mathbb{R}^{z}| (d,h)\in T_{\Gr F}(\bar{u},\bar{v})\}, d \in
\mathbb{R}^{n}
\]
is called the {\em graphical derivative} of $F$ at $(\bar{u},\bar{v})$;
\item[(ii)]
 The multifunction $\hat D^\ast F(\bar{u},\bar{v} ):
 \mathbb{R}^{z}\rightrightarrows\mathbb{R}^{n}$, defined by
\[
\hat D^\ast F(\bar{u},\bar{v} )(v^\ast):=\{u^\ast\in \mathbb{R}^{n} | (u^\ast,- v^\ast)\in \hat
N_{\Gr F}(\bar{u},\bar{v} )\}, v^\ast\in \mathbb{R}^{z}
\]
is called the {\em regular (Fr\'echet) coderivative} of $F$ at $(\bar{u},\bar{v} )$.
\item [(iii)]
 The multifunction $D^\ast F(\bar{u},\bar{v} ): \mathbb{R}^{z}\rightrightarrows\mathbb{R}^{n}$,
 defined by
\[
D^\ast F(\bar{u},\bar{v} )(v^\ast):=\{u^\ast\in \mathbb{R}^{n} | (u^\ast,- v^\ast)\in N_{\Gr
F}(\bar{u},\bar{v} )\}, v^\ast\in \mathbb{R}^{z}
\]
is called the {\em limiting (Mordukhovich) coderivative} of $F$ at $(\bar{u},\bar{v} )$.
\item [(iv)]
 Given a pair of directions $(d,h) \in \mathbb{R}^{n} \times \mathbb{R}^{z}$, the
 multifunction\\
 $D^\ast F((\bar{u},\bar{v} ); (d,h)):
 \mathbb{R}^{n}\rightrightarrows\mathbb{R}^{z}$, defined by
\begin{equation*}
D^\ast  F((\bar{u},\bar{v} ); (d,h))(v^\ast):=\{u^\ast \in \mathbb{R}^{n} | (u^\ast,-v^\ast)\in
N_{\Gr F}((\bar{u},\bar{v} ); (d,h)) \}, v^\ast\in \mathbb{R}^{z}
\end{equation*}
is called the {\em directional limiting coderivative} of $F$ at $(\bar{u},\bar{v} )$ in
direction $(d,h)$.
\end{enumerate}
\end{definition}

For the properties of the cones $T_A(\bar{x})$, $\hat N_A(\bar{x})$ and $N_A(\bar{x})$ from Definition
\ref{DefVarGeom} and generalized derivatives (i), (ii) and (iii) from Definition \ref{DefGenDeriv} we
refer the interested reader to the monographs \cite{RoWe98} and \cite{M1}. The directional limiting normal cone and coderivative were introduced by the first author in \cite{Gfr13a} and various properties of these objects
can be found also in \cite{GO3} and the references
therein. Note that $D^\ast  F(\bar{u},\bar{v} )=D^\ast  F((\bar{u},\bar{v} ); (0,0))$ and that
$\dom D^\ast  F((\bar{u},\bar{v} ); (d,h))=\emptyset$ whenever $h\not\in DF(\bar{u},\bar{v})(d)$.

The above notions enable us to come back to the {\em solution map} (\ref{eq-200}) and state the (already announced) sufficient condition for the Aubin property of $S$ around $(\bar{p},\bar{x})$ from \cite{GO3}.

\begin{theorem}\label{Thm.1}
 \cite[Theorem 4.4]{GO3}. Let $M$ have a closed graph and assume that
 \begin{enumerate}
  \item [(i)]
  \[
  \{k \in \mathbb{R}^{n} | 0 \in DM(\bar{p},\bar{x},0)(h,k)\} \neq \emptyset \mbox{ for all } h \in \mathbb{R}^{m};
  \]
   \item [(ii)]
 $M$ is metrically subregular at $(\bar{p},\bar{x},0)$;
 \item [(iii)]
 For every nonzero $(h,k)\in \mathbb{R}^{m} \times \mathbb{R}^{n}$ verifying $0 \in
 DM(\bar{p},\bar{x},0)(h,k)$ one has the implication
 \begin{equation*}
(q^{*},0)\in D^{*}M ((\bar{p},\bar{x},0); (h,k,0))(v^{*})\Rightarrow q^{*}=0.
 \end{equation*}
 \end{enumerate}
 Then $S$ has the Aubin property around $(\bar{p},\bar{x})$ and for any $h \in \mathbb{R}^{m}$
 \begin{equation*}
DS(\bar{p},\bar{x})(h)=\{k|0\in DM(\bar{p},\bar{x},0)(h,k)\}.
 \end{equation*}
 The above assertions remain true provided assumptions (ii), (iii) are replaced by
 \begin{enumerate}
\item [(iv)] For every nonzero $(h,k)\in \mathbb{R}^{m} \times \mathbb{R}^{n}$ verifying $0 \in DM
    (\bar{p},\bar{x},0)(h,k)$ one has the implication
\begin{equation*}
(q^{*},0)\in D^{*}M ((\bar{p},\bar{x},0); (h,k,0))(v^{*})\Rightarrow
\left \{
\begin{array}{l}
q^{*}=0\\ v^{*}=0.
\end{array}
\right.
 \end{equation*}
 \end{enumerate}
\end{theorem}

We are now ready to proceed to the proper problem formulation. As announced in the Introduction, this paper is devoted to solution maps of a class of variational systems in which
\begin{equation}\label{eq-50}
M(p,x):= f(p,x)+\hat{N}_{\Gamma(p,x)}(x),
\end{equation}
with $f: \mathbb{R}^{m} \times \mathbb{R}^{n} \rightarrow \mathbb{R}^{n}$ being continuously differentiable and $\Gamma : \mathbb{R}^{m} \times \mathbb{R}^{n} \rightrightarrows\mathbb{R}^{n}$ given via
\begin{equation}\label{eq-51}
\Gamma(p,x)=\{y \in \mathbb{R}^{n} | q(p,x,y)\in D\}.
\end{equation}
In (\ref{eq-51}), $q: \mathbb{R}^{m} \times \mathbb{R}^{n} \times \mathbb{R}^{n}\rightarrow \mathbb{R}^{s}$ is twice continuously differentiable and $D \subset \mathbb{R}^{s} $ is convex and polyhedral.

Consider the reference point $(\bar{p},\bar{x})$ from the graph of the solution map $S$ and, to unburden the notation, let us introduce the functions $\tilde{q}: \mathbb{R}^{m} \times \mathbb{R}^{n} \rightarrow \mathbb{R}^{s}$ and $b : \mathbb{R}^{m} \times \mathbb{R}^{n} \rightarrow \mathbb{R}^{s \times n}$ by
\[
\tilde{q}(p,x)= q(p,x,x)~ \mbox{ and }~ b(p,x)=\nabla_{3}q(p,x,x).
\]
Throughout the whole paper we will impose the following assumption:
\begin{enumerate}
 \item [(A)]
  The implication
  \begin{equation}\label{eq-52}
\left.\begin{aligned}
& b(\bar{p},\bar{x})^{T} \lambda =0\\
& \lambda \in {\rm sp\,} N_{D}(\tilde{q}(\bar{p},\bar{x}))
\end{aligned}\right\} \Rightarrow \lambda = 0
  \end{equation}
  is fulfilled.
  \end{enumerate}
(A)
  entails in particular that the {\em generalized equation} (GE)
  \begin{equation}\label{eq-400}
  0 \in f(p,x)+\hat{N}_{\Gamma(p,x)}(x)
  \end{equation}
is locally, around $(\bar{p},\bar{x})$, equivalent with the (possibly simpler) GE
\begin{equation}\label{eq-53}
0 \in f(p,x)+ b(p,x)^{T}N_{D}(\tilde{q}(p,x))
\end{equation}
which will be used as our basic model in the whole development. Indeed, as argued in \cite{QVI}, this follows from a slight modification of amenability results in \cite[Chapter 10.F]{RoWe98} when applied to the set $\Gamma(\bar{p},\bar{x})$ at $\bar{x}\in \Gamma(\bar{p},\bar{x})$. In fact, this equivalence holds true even under a relaxation of (A), where  the second line on  the left-hand side of (\ref{eq-52}) is replaced by $\lambda \in N_{D}(\tilde{q}(\bar{p},\bar{x}))$. Note that this relaxed condition is imposed in \cite{QVI} instead of (A).

Since $D$ is polyhedral, (A) is equivalent with the {\em nondegeneracy} of $\Gamma(\bar{p},\bar{x})$ at $\bar{x}$, i.e., with the condition
\begin{equation*}
b(\bar{p},\bar{x})\mathbb{R}^{n} + \lin T_{D} (\tilde{q}(\bar{p},\bar{x}))=\mathbb{R}^{s}.
\end{equation*}
This follows from \cite[formula (4.172) and Example 3.139]{BoSh}. The polyhedrality of $D$ implies further that we can employ the efficient representation of  $T_{\Gr N_{D}}$ and its polar provided in  \cite[Section 2]{DoRo96}.

Finally  note that, given a $y^* \in \hat{N}_{\Gamma(\bar{p},\bar{x})}(\bar{x})$, under
 (A)  the relations
\begin{equation}\label{eq-55}
y^* = b(\bar{p},\bar{x})^{T}\lambda, ~ \lambda \in N_{D}(\tilde{q}(\bar{p},\bar{x}))
\end{equation}
 have a unique solution $\lambda$. Thanks to this fact, most formulas in the sequel are substantially simplified.

To derive the announced new criterion for the Aubin property of solution maps given by  (\ref{eq-200}) and (\ref{eq-50}), we will in the first step apply Theorem \ref{Thm.1} to GE (\ref{eq-53}). The needed graphical derivative and directional limiting coderivative of the respective mapping $M$ are computed in the next two sections.

\section{Computation of the graphical derivative}

The right-hand side of (\ref{eq-53}) amounts to the sum of a smooth single-valued function $f$ and the multifunction $Q:\mathbb{R}^{m} \times \mathbb{R}^{n} \rightrightarrows \mathbb{R}^{n}$
defined via
\[
Q(p,x):=b(p,x)^{T}N_{D}(\tilde{q}(p,x)).
\]
The graphical derivative of $Q$ is related with the one of the mapping
$\Psi:\R^m\times\R^n\times\R^n\rightrightarrows\R^n$ given by
\begin{equation*}
  \Psi(p,x,y):=\hat N_{\Gamma(p,x)}(y).
\end{equation*}
Note that $ Q(\pb,\xb)=\Psi(\pb,\xb,\xb)$. In what follows we denote $\zb:=(\pb,\xb,\xb)$ and for any $z^*=(p^*,x^*,y^*)\in\R^m\times\R^n\times\R^n$ we denote by $\pi_3$ the canonical projection of $z^\ast$ on its third component, i.e.,
$\pi_3(z^*)=y^*$.

\begin{proposition}\label{PropGraphDer}
  Under assumption (A) for all $y^*\in  \Psi(\zb)$ and all
  $w=(h,k,l)$ we have
  \begin{align}\label{EqGraphDerPsi1}
    D\Psi(\zb,y^*)(w)&=\nabla(\nabla_3q(\cdot)^T\lambda)(\zb)w+\pi_3(N_{\K_{\gph\Gamma}(\zb,\nabla
    q(\zb)^T\lambda)}(w))\\
    \label{EqGraphDerPsi2}&=\nabla(\nabla_3q(\cdot)^T\lambda)(\zb)w+\nabla_3q(\zb)^T
    N_{\K_D(q(\zb),\lambda)}(\nabla q(\zb)w),
  \end{align}
  where $\lambda$ is the unique solution of the system
  \begin{equation}\label{EqLambda}\nabla_3q(\zb)^T\lambda=y^*,\ \lambda\in
  N_D(q(\zb)).\end{equation}
\end{proposition}
\begin{proof}
  Assumption (A) implies  the weaker condition $\nabla_3q(\pb,\xb,\xb)^T\mu=0, \mu\in
  N_D(q(\pb,\xb,\xb))\Rightarrow \mu=0$ which in turn is equivalent with the metric regularity of
  the mapping $y\rightrightarrows q(\pb,\xb,y)-D$ around $(\xb,0)$, see \cite[Example
  9.44]{RoWe98}. Hence, by \cite[Corollary 3.7]{GfrMo17a} we deduce that the system $q(p,x,y)\in D$
  enjoys the so-called Robinson stability property at $(\pb,\xb,\xb)$, i.e., there is a constant
  $\kappa>0$ together with neighborhoods $V$ of $\xb$ and $W$ of $(\pb,\xb)$ such that
  \[\dist{y,\Gamma(p,x)}\leq \kappa \dist{q(p,x,y),D}\ \forall y\in V,\;(p,x)\in W.\]
  Because $D$ is  convex and polyhedral, we can apply \cite[Theorem 5.3]{GfrMo17} to
  compute the graphical derivative $D\Psi(\zb,y^*)(w)$ resulting in \eqref{EqGraphDerPsi1}.
  Since $N_{\K_{\gph\Gamma}(z,\nabla q(z)^T\lambda)}(w)=\Big(\nabla
  q(z)^T\big(N_D(q(z))+[\lambda]\big)\Big)\cap [w]^\perp$, we have
\[\pi_3(N_{\K_{\gph\Gamma}(\zb,\nabla q(\zb)^T\lambda)}(w))=\{\nabla_3 q(\zb)^T\eta\mv \eta\in
N_D(q(\zb))+[\lambda], \eta^T\nabla q(\zb)w=0\}.\]
Next, by using the identity
\begin{align*}\big(N_D(q(\zb))+[\lambda]\big)\cap [\nabla
q(\zb)w]^\perp&=\big(T_D(q(\zb))\cap[\lambda]^\perp\big)^\circ\cap [\nabla q(\zb)w]^\perp\\
&=\K_D(q(\zb),\lambda)^\circ \cap [\nabla q(\zb)w]^\perp=N_{\K_D(q(\zb),\lambda)}(\nabla
q(\zb)w),\end{align*} we obtain \eqref{EqGraphDerPsi2} and the proof is complete.
\end{proof}

\begin{remark}\label{Rem.1}
Since $N_{\K_D(q(\zb),\lambda)}(\nabla q(\zb)w)\subset N_D(q(\zb))+[\lambda]\subset {\rm sp\,} N_D\big(q(\zb)\big)$, for every  $y^*\in\Psi(\zb)$, every direction
$w=(h,k,l)\in\R^m\times\R^n\times\R^n$ and every $v\in D\Psi(\zb,y^*)(w)$ there is a unique element
$\eta$ satisfying
\begin{equation}
  \label{EqEta}\nabla_3q(\zb)^T\eta=v-\nabla(\nabla_3q(\cdot)^T\lambda)(\zb)w,\ \eta \in
  N_{\K_D(q(\zb),\lambda)}(\nabla q(\zb)w).
\end{equation}
\end{remark}
We are now ready to compute the graphical derivative of $Q$.
\begin{theorem}
  \label{ThGraphDer}
  For all $y^*\in Q(\pb,\xb)$  and all $(h,k)\in\R^m\times\R^n$ we have
  \begin{equation}\label{EqInclGraphDer}DQ((\pb,\xb),y^*)(h,k)\subset
  D\Psi(\zb,y^*)(h,k,k).\end{equation}
  Conversely, for all $y^*\in \Psi(\zb)$, all $(h,k)\in\R^m\times\R^n$ and all $v\in
  D\Psi(\zb,y^*)(h,k,k)$ such that the mapping
  \begin{align}\label{EqDirNonDegen}F(p,x,\mu):=\big(\tilde q(p,x), \mu\big)-\gph N_D
  \end{align}
  is  metrically subregular at $\big((\pb,\xb,\lambda),0\big)$ in direction $(h,k,\eta)$ with
  $\lambda$ and $\eta$ given by \eqref{EqLambda} and \eqref{EqEta} with $w:=(h,k,k)$,
  respectively,
  we have $v\in DQ\big((\pb,\xb),y^*\big)(h,k).$
\end{theorem}
\begin{proof}The inclusion \eqref{EqInclGraphDer} follows immediately from the definition of the
graphical derivative and there remains to show the second statement. Consider $v\in
D\Psi(\zb,y^*)(h,k,k)$ such that $F$ is  metrically subregular at
$\big((\pb,\xb,\lambda),0\big)$ in direction $(h,k,\eta)$. Then there are sequences
$t_\nu\downarrow 0$, $w_\nu:=(h_\nu,k_\nu,l_\nu)\to w$ and $v_\nu\to v$ such that
\[y^*+t_\nu v_\nu \in \nabla_3q(\zb+t_\nu w_\nu)^TN_D\big(q(\zb+t_\nu w_\nu)\big) \mbox{ for all } \nu.\]
Due to (A) there is for all $\nu$ sufficiently large a unique multiplier $\lambda_\nu\in
N_D\big(q(\zb+t_\nu w_\nu)\big)$ satisfying $y^*+t_\nu v_\nu\in \nabla_3q(\zb+t_\nu
w_\nu)^T\lambda_\nu$. The sequence $\lambda_\nu$ is uniformly bounded yielding, together with
$y^*=\nabla_3 q(\zb)^T\lambda$, that
\begin{align*}t_\nu v_\nu&= \nabla_3q(z+t_\nu w_\nu)^T\lambda_\nu-y^*=\Big(\nabla_3
q(\zb)+t_\nu\nabla\big(\nabla_3 q(\zb)\big)w_\nu\Big)^T\lambda_\nu-\nabla_3
q(\zb)^T\lambda+\oo(t_\nu)\\
&=\nabla_3 q(\zb)^T(\lambda_\nu-\lambda)+t_\nu\nabla\big(\nabla_3
q(\cdot)^T\lambda_\nu\big)(\zb)w_\nu+\oo(t_\nu)\end{align*} and, consequently,
\[\lim_{\nu\to\infty}\nabla_3q(\zb)^T\frac{\lambda_\nu-\lambda}{t_\nu}=\lim_{\nu\to\infty}\Big(v_\nu-\nabla\big(\nabla_3
q(\cdot)^T\lambda_\nu\big)(\zb)w_\nu +\frac{\oo(t_\nu)}{t_\nu}\Big)=v-\nabla\big(\nabla_3
q(\cdot)^T\lambda\big)(\zb)w.\]
Since $D$ is a convex polyhedral set, we have $\lambda_\nu\in N_D\big(q(\zb+t_\nu
w_\nu)\big)\subset N_D\big(q(\zb)\big)$ for all $\nu$ sufficiently large and therefore
$\frac{\lambda_\nu-\lambda}{t_\nu}\in N_D\big(q(\zb)\big)+[\lambda]\subset {\rm
sp\,}N_D\big(q(\zb)\big)$. By virtue of (A) we conclude that $\frac{\lambda_\nu-\lambda}{t_\nu}$
is convergent to $\eta$. Since
\begin{align*}\lefteqn{\distb{(\tilde q(\pb+t_\nu h_\nu, \xb+t_\nu k_\nu),\lambda_\nu),\gph
N_D}}\\
&\leq \norm{\tilde q(\pb+t_\nu h_\nu, \xb+t_\nu k_\nu)-q(\zb+t_\nu w_\nu)}+ \distb{(q(\zb+t_\nu
w_\nu),\lambda_\nu),\gph N_D}\\ &=\norm{ q(\pb+t_\nu h_\nu, \xb+t_\nu k_\nu,x+t_\nu
k_\nu)-q(\pb+t_\nu h_\nu, \xb+t_\nu k_\nu,\xb+t_\nu
l_\nu)}=\OO(t_\nu\norm{k_\nu-l_\nu})=\oo(t_\nu),\end{align*} by the assumed directional metric
subregularity we can find for every $\nu$ sufficiently large some $p_\nu,x_\nu, \tilde \lambda_\nu$
with $0\in F(p_\nu,x_\nu,\tilde\lambda_\nu)$ and $\norm{p_\nu-(\pb+t_\nu
h_\nu)}+\norm{x_\nu-(\xb+t_\nu k_\nu)}+\norm{\tilde\lambda_\nu-\lambda_\nu}=\oo(t_\nu)$. Thus
\begin{align*}y^*+t_\nu v_\nu&=\nabla_3 q(\pb+t_\nu h_\nu,\xb_\nu+t_\nu k_\nu, \xb+t_\nu
l_\nu)^T\lambda_\nu=\nabla_3 q(p_\nu,x_\nu,x_\nu)^T\tilde\lambda_\nu+\oo(t_\nu)\\
&=b(p_\nu,x_\nu)^T\tilde\lambda_\nu+\oo(t_\nu).
\end{align*}
This equality, together with $\tilde\lambda_\nu\in N_D\big(\tilde q(p_\nu,x_\nu)\big)$, implies the inclusion $v\in
DQ\big((\pb,\xb),y^*\big)(h,k)$ and we are done.
\end{proof}
To ensure the directional metric subregularity of (\ref{EqDirNonDegen}) we may use the sufficient condition presented in Proposition \ref{LemDirSubreg} below. Recall that $\mathcal{F}$ is a {\em face} of a polyhedral convex cone $K$ provided for some vector $z^{*}\in K^{\circ}$ one has
\[
\mathcal{F}= K \cap \left[z^{*}\right]^{\perp}.
\]

\begin{proposition}\label{LemDirSubreg}
  Let $\lambda\in N_D\big(\tilde q(\pb,\zb)\big)$, let $(h,k)\in\R^m\times\R^n$ be a pair of directions
  satisfying $\nabla \tilde q(\pb,\xb)(h,k)\in \K_D\big(\tilde q(\pb,\xb),\lambda\big)$ and let
  $\eta\in N_{\K_D(\tilde q(\pb,\xb),\lambda)}\big(\nabla \tilde q(\pb,\xb)(h,k)\big)$. Further
  assume that for  every pair of faces $\F_1,\F_2$ of the critical cone $\K_D\big(\tilde
  q(\pb,\xb),\lambda\big)$ with $\nabla \tilde q(\pb,\xb)(h,k)\in \F_2\subset
  \F_1\subset[\eta]^\perp$ there holds
  \begin{equation*}    
    \nabla \tilde q(\pb,\xb)^T\mu=0, \mu\in (\F_1-\F_2)^\circ\
    \Rightarrow\ \mu=0.
  \end{equation*}
  Then the mapping $F$ given by \eqref{EqDirNonDegen} is  metrically subregular at
  $\big((\pb,\xb,\lambda),0\big)$ in direction $(h,k,\eta)$.
\end{proposition}
\begin{proof}We claim that $F$ is even metrically regular at $\big((\pb,\xb,\lambda),0\big)$ in
direction $\big((h,k,\eta),0\big)$. In order to show this claim we invoke the characterization of
directional  metric regularity from \cite[Theorem 1]{GfrKl16}, which reads in our case as
\[\nabla \tilde q(\pb,\xb)^T\mu=0,\ \xi=0,\ (\mu, \xi)\in N_{\gph N_D}\Big(\big(\tilde
q(\pb,\xb),\lambda\big);\big(\nabla \tilde q(\pb,\xb)(h,k),\eta\big)\Big)\ \Rightarrow\ \mu=0,\
\xi=0.\]
By \cite[Theorem 2.12]{GO3}, $N_{\gph N_D}\Big(\big(\tilde q(\pb,\xb),\lambda\big);\big(\nabla \tilde
q(\pb,\xb)(h,k),\eta\big)\Big)$ amounts to the union of all product sets $K^\circ\times K$ associated with
cones $K$ of the form $\F_1-\F_2$, where $\F_1,\F_2$ are faces of the critical cone
$\K_D\big(\tilde q(\pb,\xb),\lambda\big)$ with $\nabla \tilde q(\pb,\xb)(h,k)\in \F_2\subset
\F_1\subset[\eta]^\perp$. Thus our claim about the directional metric regularity of $F$ holds true
and the statement is proved.
\end{proof}

Of course, for the verification of the directional metric subregularity of (\ref{EqDirNonDegen}) one could employ also some non-directional less fine criteria mentioned, e.g., in \cite{QVI} and \cite{Hen}.

To write down the final formula for the  graphical derivative of $M$, we associate now with the considered variational system for fixed $\lambda\in \R^s$  the {\em Lagrangian} mapping $\Lag_\lambda:\R^m\times\R^n\to \R^n$ via
\[\Lag_\lambda(p,x):=f(p,x)+b(p,x)^T\lambda.\]
\if{$\mathcal{L}: \mathbb{R}^{m} \times \mathbb{R}^{n} \times\mathbb{R}^{s} \rightarrow\mathbb{R}^{n}$ as in \cite{QVI} via
\[
\mathcal{L} (p,x,\lambda) := f(p,x)+b(p,x)^{T}\lambda.
\]}\fi
Under the assumptions of Theorem \ref{ThGraphDer} we then obtain the formula
 \begin{equation*}
DM(\bar{p},\bar{x},0)(h,k)=
\Lag_\lb (\bar{p},\bar{x}) (h,k) +b(\bar{p},\bar{x})^{T}
N_{\mathcal{K}_{D}(\tilde{q}(\bar{p},\bar{x}),\lb)} (\nabla \tilde{q}(\bar{p},\bar{x})(h,k) ).
 \end{equation*}
 where $\lb$ is the unique solution of the system
 \begin{equation}\label{eq-61}
f(\bar{p},\bar{x})+b(\bar{p},\bar{x})^{T}\lambda=0,\ \lambda \in N_{D}(\tilde{q}(\bar{p},\bar{x})). \end{equation}

 \section{Computation of the directional limiting coderivative}
Given a pair of directions $(h,k)\in  \mathbb{R}^{m} \times\mathbb{R}^{n}$, the aim of this section is to provide possibly sharp estimates of the sets $D^{*}M ((\bar{p},\bar{x},0); (h,k,0))(z^{*})$. Due to \cite[formula (2.4)]{GO3} and the local equivalence of GEs (\ref{eq-400}) and (\ref{eq-53}) we have for any $v^*\in\R^n$ the equality
\begin{equation}\label{eq-800}
D^{*}M ((\bar{p},\bar{x},0); (h,k,0))  (v^{*}) = \nabla f(\bar{p},\bar{x})^{T}v^{*}+
 D^{*}Q (\bar{p},\bar{x}, -f(\bar{p},\bar{x})); (h,k, -\nabla f(\bar{p},\bar{x})(h,k))(v^{*}).
\end{equation}
It suffices thus to compute just the directional limiting coderivative of $Q$. To this purpose we
observe that
$Q(p,x)= S_{2}\circ S_{1}(p,x)$, where $S_{1}:\mathbb{R}^{m} \times \mathbb{R}^{n} \rightrightarrows \mathbb{R}^{m} \times\mathbb{R}^{n} \times \mathbb{R}^{s}$ is given by
\[
S_{1}(p,x):=
\left[ \begin{array}{c}
p\\
 x\\
 N_{D}(\tilde{q}(p,x))
\end{array}\right ]
\]
and $S_{2}: \mathbb{R}^{m} \times\mathbb{R}^{n} \times \mathbb{R}^{s} \rightarrow \mathbb{R}^{n}$ is given by
\[
S_{2}(u_{1},u_{2},u_{3}):= b(u_{1},u_{2})^{T} u_{3}.
\]
Consider the {\em intermediate} mapping $\Xi : \mathbb{R}^{m} \times \mathbb{R}^{n} \times \mathbb{R}^{n}\rightrightarrows \mathbb{R}^{m} \times\mathbb{R}^{n} \times\mathbb{R}^{s}$ defined by
\[
\begin{split}
& \Xi (p,x,y^*):=\{(u_{1},u_{2},u_{3}) \in S_{1}(p,x)| y^* = S_{2} (u_{1},u_{2},u_{3})\}=\\
& \{(u_{1},u_{2},u_{3})|u_{1}=p, u_{2}=x, u_{3}\in N_{D}(\tilde{q}(p,x)), b(p,x)^{T}u_{3}=y^*\}.
\end{split}
\]
\begin{lemma}\label{Lem.1}
Let $y^*\in Q(\pb,\xb)$. Then
$\Xi (\bar{p},\bar{x},y^*) = \{(\bar{p},\bar{x},\lambda)\}$ with $\lambda$ (uniquely) given by (\ref{eq-55}). Moreover, the values $\Xi (p,x,v^*)$ are bounded for all $(p,x,v^*) \in \dom \Xi$ close to $(\bar{p},\bar{x},y^*)$.
\end{lemma}

\proof
The first statement is directly implied by (A); see the mention at the end of Section 2.   The boundedness follows by a standard argumentation even from a relaxed condition
\begin{equation*}
\left. \begin{aligned}
& b(\bar{p},\bar{x})^{T}\lambda = 0 \\
& \lambda \in N_{D}(\tilde{q}(\bar{p},\bar{x}))
\end{aligned}\right \} \Rightarrow \lambda = 0,
\end{equation*}
see \cite[p.18]{QVI}.
\endproof

\begin{lemma}\label{Lem.2}
Let $\bar{d}=(\bar{p},\bar{x},\lambda) = \Xi (\bar{p},\bar{x},y^*)$. Then  the set
\begin{equation*}
\{\xi \in \mathbb{S}|\xi \in D S_{1}(\bar{p},\bar{x},\bar{d})(0), ~ 0 = \nabla S_{2}(\bar{d},y^*)(\xi)\}
\end{equation*}
 is empty.
\end{lemma}

\proof
Clearly,
\begin{equation}\label{eq-205}
DS_{1}(\bar{p},\bar{x},\bar{d})(0)=\{(0,0,\eta) \in \mathbb{R}^{m} \times \mathbb{R}^{n} \times \mathbb{R}^{s}|\eta \in D(N_{D}\circ \tilde{q})(\bar{p},\bar{x},\lambda)(0,0)\}
\end{equation}
and the condition $0 = \nabla S_{2}(\bar{d},y^*)(\xi)$ amounts to
\begin{equation}\label{eq-206}
0 = \nabla (b(\bar{p},\bar{x})^{T}\lambda)
\left[ \begin{array}{c}
\xi_{1}\\
\xi_{2}
\end{array}\right] +
 b(\bar{p},\bar{x})^{T}\xi_{3}.
\end{equation}
By comparing (\ref{eq-205}) and (\ref{eq-206}) it follows directly that $\xi_{1}=0, \xi_{2}=0$ and it remains to show that the conditions $\eta \in D(N_{D}\circ \tilde{q})(\bar{p},\bar{x},\lambda)(0,0), 0 = b(\bar{p},\bar{x})^{T}\eta$ imply $\eta = 0$. Clearly,
\[
T_{\gph( N_{D}\circ \tilde{q})}(\bar{p},\bar{x},\bar{\lambda})\subset \left\{(h,k,\eta) \left| \left[\begin{array}{c}
\nabla \tilde{q}(\bar{p},\bar{x})(h,k)\\
\eta
\end{array}\right] \in T_{\gph N_{D}}(\tilde{q}(\bar{p},\bar{x}),\lambda)
\right. \right\}.
\]
It follows that $\eta \in D( N_{D}\circ \tilde{q})(\bar{p},\bar{x},\lambda)(0,0)$ implies
 that $(0,\eta) \in T_{\gph N_{D}}(\tilde{q}(\bar{p},\bar{x}),\lambda)$. Due to the polyhedrality of $D$, one has (cf. \cite[page 1093]{DoRo96})
\[
T_{\gph N_{D}}(\tilde{q}(\bar{p},\bar{x}),\lambda)=\{(a,b)\in \mathbb{R}^{s} \times \mathbb{R}^{s}| a \in
\mathcal{K}_{D} (\tilde{q}(\bar{p},\bar{x}),\lambda), b \in
 \mathcal{K}_{D} (\tilde{q}(\bar{p},\bar{x}), \lambda)^{\circ}, \langle a,b\rangle=0 \},
\]
from which we infer that
 $\eta \in N_{D}(\tilde{q}(\bar{p},\bar{x}))+[\lambda] \subset {\rm sp}~ N_{D}(\tilde{q}(\bar{p},\bar{x}))$.  Consequently, $\xi_{3}= \eta = 0$ by virtue of (A) and we are done.
\endproof

As the last auxiliary result we will now estimate the directional limiting coderivative of $S_{1}$. Clearly, $S_{1}=\Sigma \circ \Omega$, where $\Omega: \mathbb{R}^{m+n} \rightarrow \mathbb{R}^{m+n} \times \mathbb{R}^{m+n} $ is defined by
\[
\Omega(p,x)=
\left[ \begin{array}{c}
(p,x)\\
(p,x)
\end{array}\right] ~ ({\rm two~ copies}),
\]
and $\Sigma : \mathbb{R}^{m+n} \times \mathbb{R}^{m+n} \rightrightarrows   \mathbb{R}^{m+n} \times \mathbb{R}^{s}$ is defined by
\[
\Sigma(a_{1},a_{2})=
\left[ \begin{array}{c}
 a_{1}\\
 (N_{D}\circ \tilde{q})(a_{2})
\end{array}\right].
\]
\begin{lemma}\label{Lem.3}
Consider a direction $(h,k,\eta)\in\mathbb{R}^{m} \times \mathbb{R}^{n} \times \mathbb{R}^{s}$ and a point $\lambda $ such that $ (\bar{p},\bar{x},\lambda)\in S_{1}(\bar{p},\bar{x})$. Then one has for any $d^{*}=(d^{*}_{1},d^{*}_{2},d^{*}_{3})\in \mathbb{R}^{m} \times \mathbb{R}^{n} \times \mathbb{R}^{s} $ the inclusion
\begin{equation}\label{eq-207}
D^{*}S_{1}((\bar{p},\bar{x},\lambda);(h,k,(h,k,\eta)))(d^{*})\subset
\left[ \begin{array}{c}
d^{*}_{1}\\
d^{*}_{2}
\end{array}\right] + D^{*}(N_{D}\circ \tilde{q})((\bar{p},\bar{x},\lambda); (h,k,\eta))(d^{*}_{3}).
\end{equation}
\end{lemma}
\proof
The statement follows from \cite[Corollary 5.1]{BGO}, provided we verify the respective subregularity condition. To this aim we observe that the implication
\begin{equation}\label{eq-208}
\left. \begin{aligned}
& 0 \in \nabla \Omega (\bar{p},\bar{x})^{T}(a^{*}_{1}, a^{*}_{2})\\
& - (a^{*}_{1}, a^{*}_{2})\in D^{*}\Sigma ((\bar{p},\bar{x}), (\bar{p},\bar{x}), (\bar{p},\bar{x},\lambda))(0,0)
\end{aligned}\right\} \Rightarrow  a^{*}_{1}=0, a^{*}_{2}=0
\end{equation}
is fulfilled. Indeed, the relations on the left-hand side of (\ref{eq-208}) imply that $a^{*}_{1}+ a^{*}_{2}=0$ and $a^{*}_{1}=0$, whence $a^{*}_{2}=0$ as well. On the other hand, implication (\ref{eq-208}) is a strengthened (non-directional) variant of condition (32) in \cite{BGO}, which ensures the subregularity condition in \cite[Corollary 5.1]{BGO}. We obtain thus that
\[
 \begin{split}
D^{*}S_{1}((\bar{p},\bar{x},\lambda);(h,k,(h,k,\eta)))(d^{*})  & \subset
\nabla \Omega (\bar{p},\bar{x})^{T}\circ D^{*}\Sigma (((\bar{p},\bar{x}), (\bar{p},\bar{x}), (\bar{p},\bar{x},\lambda));\\
& ((h,k),(h,k),(h,k,\eta)))(d^{*}),
\end{split}
\]
which directly leads to inclusion (\ref{eq-207}).
\endproof
We are now in position to compute an estimate of the directional limiting coderivative of $Q$ at the point  $(\bar{p},\bar{x},\bar{y}^{*})$ in the direction $(h,k,l)$.
\begin{theorem}\label{Thm.3}
Let $y^*\in\R^n$ be given and let $\lambda \in \mathbb{R}^{s}$ be (uniquely) given by the relations (\ref{eq-55})
and $\eta \in \mathbb{R}^{s}$ be (uniquely) given by
 \begin{equation}\label{eq-700}
l = (\nabla(b(\bar{p},\bar{x})^{T} \lambda) (h,k) + b(\bar{p},\bar{x})^{T}\eta, ~\eta \in ~N_{\mathcal{K}(\tilde{q}(\bar{p},\bar{x}),\lambda)}(\nabla \tilde{q}(\bar{p},\bar{x})(h,k)).
 \end{equation}
  Assume that the mapping (\ref{EqDirNonDegen}) is metrically subregular at $((\bar{p},\bar{x},\lambda),0)$ in direction $(h,k,\eta)$.

 Then for any $v^{*}\in \mathbb{R}^{n}$ one has the estimate
\begin{multline}\label{eq-209}
D^{*}Q((\bar{p},\bar{x},y^*);(h,k,l))  (v^{*}) \subset \nabla(b(\bar{p},\bar{x})^{T} \lambda)^{T}v^{*} +\\
 \nabla \tilde{q}(\bar{p},\bar{x})^{T}D^{*}N_{D}((\tilde{q}(\bar{p},\bar{x}),\lambda);(\nabla \tilde{q}(\bar{p},\bar{x})(h,k),\eta))(b(\bar{p},\bar{x})v^{*}).
\end{multline}

\end{theorem}
\proof
We observe first that by virtue of (A) and Lemma \ref{Lem.1} all assumptions of \cite[Corollary 5.2]{BGO} are fulfilled and, thanks to Lemma \ref{Lem.1} and  Lemma \ref{Lem.2}, the inclusion in \cite[formula (26)]{BGO} simplifies to
\begin{equation}\label{eq-210}
D^{*}Q ((\bar{p},\bar{x},y^*);(h,k,l))\subset
\bigcup\limits_{\scriptstyle \xi \in D S_{1}(\bar{p},\bar{x},\bar{u})(h,k) \atop                   \scriptstyle l= \nabla S_{2}(\bar{u})\xi}
D^{*}S_{1}((\bar{p},\bar{x},\bar{u});(h,k,\xi))\circ \nabla S_{2}(\bar{u})^{T},
\end{equation}
where $\bar{u}=(\bar{p},\bar{x}, \lambda)$. The directional limiting coderivative of $S_{1}$ has been estimated in Lemma \ref{Lem.3} and so we compute now the graphical derivative of $S_{1}$ and the Jacobian of $S_{2}$.
Since
\begin{equation}\label{eq-600}
(\bar{p},\bar{x},\lambda) \in \gph (N_{D}\circ \tilde{q})~~
\Leftrightarrow~~
(\tilde{q}(\bar{p},\bar{x}),\lambda) \in \gph N_{D},
\end{equation}
we have $(h,k,\eta)\in T_{\gph( N_{D}\circ \tilde{q})}(\bar{p},\bar{x},\lambda)$ if and only if there are sequences $t_\nu\downarrow 0$, $(h_\nu,k_\nu,\eta_\nu)\to (h,k,\eta)$ such that
\[\big(\tilde q(\pb+t_\nu h_\nu,\xb+t_\nu k_\nu),\lambda+t_\nu \eta_\nu\big)\in\gph N_D\ \mbox{ for all }\nu\]
and it follows that
\[
T_{\gph( N_{D}\circ \tilde{q})}(\bar{p},\bar{x},\lambda)\subset R:=\left\{(h,k,\eta) \left|
\left[\begin{array}{c}
\nabla \tilde{q}(\bar{p},\bar{x})(h,k)\\
\eta
\end{array}\right]
 \in T_{\gph N_{D}}(\tilde{q}(\bar{p},\bar{x}),\lambda)
 \right.\right\}.
\]
On the other hand, given $(h,k,\eta)\in R$, there is a sequence $t_\nu\downarrow 0$ such that
\[\dist{(\tilde q(\pb,\xb)+t_\nu\nabla \tilde q(\pb,\xb)(h,k),\lambda+t_\nu \eta), \gph N_D}=\oo(t_\nu).\]
Hence, $\dist{(\tilde q(\pb +t_\nu h,\xb+t_\nu k),\lambda+t_\nu \eta), \gph N_D}=\oo(t_\nu)$ and we can employ the metric subregularity of (\ref{EqDirNonDegen})
 at $((\bar{p},\bar{x},\lambda),0)$ to obtain $(p_\nu,x_\nu,\lambda_\nu)$ satisfying
 \[\norm{(p_\nu,x_\nu,\lambda_\nu)-(\pb +t_\nu h,\xb+t_\nu k,\lb+t_\nu \eta)}=\oo(t_\nu), \ \big(\tilde q(p_\nu,x_\nu),\lambda_\nu\big)\in\gph N_D.\]
 We conclude $(h,k,\eta)\in T_{\gph( N_{D}\circ \tilde{q})}(\bar{p},\bar{x},\lambda)$ and $R\subset T_{\gph( N_{D}\circ \tilde{q})}(\bar{p},\bar{x},\lambda)$ follows.
 Hence\\$T_{\gph( N_{D}\circ \tilde{q})}(\bar{p},\bar{x},\lambda)=R$   and
\if{we can employ the metric subregularity of (\ref{EqDirNonDegen})
 at $((\bar{p},\bar{x},\bar{\lambda}),0)$ and obtain that
\[
T_{\gph( N_{D}\circ \tilde{q})}(\bar{p},\bar{x},\bar{\lambda})= \left\{(h,k,\eta) \left|
\left[\begin{array}{c}
\nabla \tilde{q}(\bar{p},\bar{x})(h,k)\\
\eta
\end{array}\right]
 \in T_{\gph N_{D}}(\tilde{q}(\bar{p},\bar{x}),\bar{\lambda})
\right. \right\}.
\]}\fi
this relation, together with \cite[Example 4A.4]{DoRo}, implies that
\[
DS_{1}(\bar{p},\bar{x},\bar{u})(h,k) = \left\{(h,k,\eta)\left| \eta \in N_{\K_D(\tilde{q}(\bar{p},\bar{x}),\lambda)}(\nabla \tilde{q}(\bar{p},\bar{x})(h,k))\right.\right\}.
\]
Further, by a simple calculation we obtain that
\[
\nabla S_{2}(\bar{u})= \left[\nabla(b(\bar{u}_{1}, \bar{u}_{2})^{T}\bar{u}_{3}, ~ b(\bar{u}_{1}, \bar{u}_{2})^{T})\right],
\]
and thus the union in (\ref{eq-210}) is taken over all $\xi = (h,k,\eta)$ satisfying (\ref{eq-700}).
The uniqueness of $\eta$ follows from the comparison of (\ref{eq-700}) with (\ref{EqEta}) and Remark \ref{Rem.1}.
From  (\ref{eq-210}) we get now the inclusion

\begin{equation}\label{eq-212}
\begin{split}
D^{*}Q ((\bar{p},\bar{x},y^*);(h,k,l))(v^{*}) & \subset (\nabla(b(\bar{p},\bar{x})^{T} ) \lambda)^{T}v^{*} +\\
& D^{*}(N_{D}\circ \tilde{q})(\bar{p},\bar{x},\lambda);(h,k,\eta))(b(\bar{p},\bar{x})v^{*}),
\end{split}
\end{equation}
and it remains to rewrite the second term on the right-hand side of (\ref{eq-212}) in a more tractable form. To this aim we invoke \cite[Corollary 3.2]{BGO}. Indeed, applying the equivalence (\ref{eq-600}) as in the computation of $DS_{1}(\bar{p},\bar{x},\bar{u})(h,k)$,
under the assumed directional metric subregularity of mapping     (\ref{EqDirNonDegen}), we obtain the inclusion
\[
\begin{split}
D^{*}(N_{D}\circ \tilde{q}) & ((\bar{p},\bar{x},\lambda);(h,k,\eta))(b(\bar{p},\bar{x})v^{*})\subset\\
& \nabla \tilde{q}(\bar{p},\bar{x})^{T}D^{*}N_{D} ((\tilde{q}(\bar{p},\bar{x}),\lambda);
 (\nabla\tilde{q}(\bar{p},\bar{x}) (h,k), \eta))(b(\bar{p},\bar{x})v^{*}),
\end{split}
\]
and the proof is complete.
\endproof

\if{
Note that the metric subregularity of  (\ref{EqDirNonDegen}) is implied
\begin{equation}\label{eq-114}
\left.\begin{array}{c}
\nabla_{p}\tilde{q}(\bar{p},\bar{x})^{T}a=0\\
\nabla_{x}\tilde{q}(\bar{p},\bar{x})^{T}a=0\\
a \in N_{\gph N_{D}}((\tilde{q}(\bar{p},\bar{x}),\bar{\lambda}); (\nabla \tilde{q}(\bar{p},\bar{x})(h,k),\eta))(0)
\end{array}\right\} \Rightarrow a = 0,
 \end{equation}
 which follows from \cite[Corollary 1]{GfrKl16},  see also \cite[Proposition 2.2]{BGO}. \green{Brauchen wir \eqref{eq-114}? Das ist genau Proposition \ref{LemDirSubreg}}
}\fi
We can now combine inclusion (\ref{eq-209}) with relation (\ref{eq-800}) to obtain a formula for \\ $D^{*}M((\bar{p},\bar{x},0);(h,k,0))$ in terms of problem data. To this  purpose we introduce $\bar{y}^{*}= - f(\bar{p},\bar{x})$ and denote by $\bar{\lambda}$ the (unique) solution of (\ref{eq-61}). Under the assumptions of Theorem \ref{Thm.3} we obtain the estimate
\[
\begin{split}
D^{*}M(\bar{p},\bar{x},0); & (h,k,0))(v^{*})\subset \nabla \mathcal{L}_{\bar{\lambda}}(\bar{p},\bar{x})^{T} v^{*}+\\
& \nabla \tilde{q}(\bar{p},\bar{x})^{T}D^{*}N_{D}((\tilde{q}(\bar{p},\bar{x})\bar{\lambda});
(\nabla \tilde{q}(\bar{p},\bar{x})(h,k),\eta))(b(\bar{p},\bar{x})v^{*})
\end{split}
\]
which will be utilized in the next section.

\section{On the Aubin property of the solution map}

Combining Theorem \ref{Thm.1} with the formulas for $DM(\bar{p},\bar{x},\bar{z})(h,k)$ and
 $D^{*}M((\bar{p},\bar{x},\bar{z});(h,k,0))$ derived in Sections 3 and 4, respectively, we arrive at the following result.
\begin{theorem}\label{Thm.4}
Assume that $\bar{\lambda}$ is the (unique) solution of system (\ref{eq-61}) and for every $h \in \mathbb{R}^{m}$ there is some $k \in\mathbb{R}^{n}$ and some $\eta \in \mathbb{R}^{s}$
such that
\begin{equation}\label{eq-500}
0 = \nabla\Lag_\lb(\bar{p},\bar{x})(h,k)+ b(\bar{p},\bar{x})^{T}\eta,~~
\eta \in N_{\mathcal{K}_{D}(\tilde{q}(\bar{p},\bar{x}),\bar{\lambda})}(\nabla \tilde{q}(\bar{p},\bar{x})(h,k)).
\end{equation}
Further assume that for every nonzero pair $(h,k)\in \mathbb{R}^{m} \times \mathbb{R}^{n}$ and the corresponding (unique) $\eta \in \mathbb{R}^{s}$ satisfying (\ref{eq-500}) the mapping $F$ given by (\ref{EqDirNonDegen}) is metrically subregular at $((\bar{p},\bar{x},\bar{\lambda}),0)$ in direction $(h,k,\eta)$ and the implication
\begin{equation}\label{eq-301}
\left.\begin{aligned}
& (p^{*},0)=\nabla\Lag_\lb(\bar{p},\bar{x})^{T}v^{*} +
\nabla \tilde{q}(\bar{p},\bar{x})^{T}w\\
& w \in D^{*}N_{D}((\tilde{q}(\bar{p},\bar{x}),\bar{\lambda}); (\nabla \tilde{q}(\bar{p},\bar{x})(h,k),\eta))(b(\bar{p},\bar{x})v )
\end{aligned}\right\} \Rightarrow p^{*}=0, v  = 0
\end{equation}
is fulfilled. Then the solution map defined via (\ref{eq-200}) and (\ref{eq-50}) has the Aubin property around $(\bar{p},\bar{x})$. Moreover, one has
\begin{equation}\label{eq-302}
DS(\bar{p},\bar{x})(h) = \left\{  k \left| 0 \in \nabla\Lag_\lb(\bar{p},\bar{x})(h,k)+ b(\bar{p},\bar{x})^{T}
 N_{\mathcal{K}_{D}(\tilde{q}(\bar{p},\bar{x}),\bar{\lambda})}(\nabla \tilde{q}(\bar{p},\bar{x})(h,k))\right.\right\}.
\end{equation}
\end{theorem}
Following Theorem \ref{Thm.1}, the implication (\ref{eq-301}) could be weakened by omitting the requirement $v  = 0$ on its right-hand side. Then, however, we have to impose an additional requirement that the mapping (\ref{eq-50}) is metrically subregular at $(\bar{p},\bar{x},0)$.

If the constraint mapping $q$ (and hence also $\tilde{q}$) does not depend on $p$, then (\ref{eq-500}) attains the form
\begin{equation*}
0 = \nabla\Lag_\lb(\bar{p},\bar{x})(h,k)+ b(\bar{x})^{T}\eta,~~
\eta \in  N_{\mathcal{K}_{D}(\tilde{q}(\bar{x}),\bar{\lambda})}(\nabla \tilde{q}(\bar{x})k)
\end{equation*}
and (\ref{eq-301}) reduces to a substantially more tractable form
\begin{equation*}
\left.\begin{aligned}
& 0= \nabla_{2}\Lag_\lb(\bar{p},\bar{x})^{T}v  +
\nabla \tilde{q}(\bar{x})^{T}w\\
& w \in D^{*}N_{D}((\tilde{q}(\bar{x}),\bar{\lambda}); (\nabla \tilde{q}(\bar{x})k,\eta))(b(\bar{x})v )
\end{aligned}\right\} \Rightarrow  v  = 0.
\end{equation*}

The polyhedrality of $D$ enables us to avoid  the computation of directional limiting coderivatives and to replace the verification of (\ref{eq-301}) by a simpler procedure. The key argument comes from the already mentioned \cite[Theorem 2.12]{GO3}.

\begin{theorem}\label{Thm.5}
In the setting of Theorem \ref{Thm.4} replace the implication (\ref{eq-301})
by the assumption that for every pair of faces $\mathcal{F}_{1}, \mathcal{F}_{2}$  of the critical cone $\mathcal{K}_{D} (\tilde{q}(\bar{p},\bar{x}),\bar{\lambda})$ with $\nabla \tilde{q}(\bar{p},\bar{x})(h,k) \in \mathcal{F}_{2} \subset \mathcal{F}_{1} \subset [\eta]^{\perp}$ there holds
\begin{equation}\label{eq-305}
 \nabla_{2}\tilde{q}(\bar{p},\bar{x})^{T}\mu = 0, \mu \in (\mathcal{F}_{1}- \mathcal{F}_{2})^{\circ} \Rightarrow \nabla_{1}\tilde{q}(\bar{p},\bar{x})^{T}\mu = 0.
\end{equation}
and for every $w \neq 0$ with $b(\bar{p},\bar{x})w \in \mathcal{F}_{1}- \mathcal{F}_{2} $ there is some $\tilde{w}$ with $  \nabla_{2}\tilde{q}(\bar{p},\bar{x})\tilde{w} \in \mathcal{F}_{1}- \mathcal{F}_{2}$ and
\begin{equation}\label{eq-306}
w^{T} \nabla_{2} \Lag_\lb(\bar{p},\bar{x})\tilde{w}>0.
\end{equation}
Then all assertions of Theorem \ref{Thm.4} remain valid.
\end{theorem}
\proof We shall show that the conditions \eqref{eq-305}, \eqref{eq-306} imply \eqref{eq-301}.
Assume on the contrary that there is some direction $(0,0)\neq (h,k)$ verifying (\ref{eq-500}) together  with some $\eta \in \mathbb{R}^{s}$ and some pair $(p^{*},v)\neq (0,0)$ satisfying
$$
(p^{*},0) \in \nabla\Lag_\lb (\bar{p},\bar{x})^{T}v +
\nabla\tilde{q}(\bar{p},\bar{x})^{T} D^{*}N_{D}((\tilde{q}(\bar{p},\bar{x}),\bar{\lambda}); (\nabla \tilde{q}(\bar{p},\bar{x})(h,k),\eta))(b(\bar{p},\bar{x})v).
$$
Next we utilize \cite[Theorem 2.12]{GO3} to find two faces $\mathcal{F}_{1}, \mathcal{F}_{2}$ of $\mathcal{K}_{D} (\tilde{q}(\bar{p},\bar{x}),\bar{\lambda})$ and $\mu \in (\mathcal{F}_{1}- \mathcal{F}_{2})^{\circ} $ satisfying $\nabla \tilde{q}(\bar{p},\bar{x})(h,k) \in \mathcal{F}_{2} \subset \mathcal{F}_{1} \subset [\eta]^{\perp}, -b(\bar{p},\bar{x})v \in \mathcal{F}_{1}- \mathcal{F}_{2}$, and
\[
(p^{*},0) = \nabla\Lag_\lb (\bar{p},\bar{x})^{T}v +
\nabla\tilde{q}(\bar{p},\bar{x})^{T}\mu.
\]
In particular, we have
\[
\nabla_{2}\Lag_\lb (\bar{p},\bar{x})^{T}v =
-\nabla_{2}\tilde{q}(\bar{p},\bar{x})^{T}\mu
\]
and
\[
p^{*}=\nabla_{1}\Lag_\lb (\bar{p},\bar{x})^{T}v +
\nabla_{1}\tilde{q}(\bar{p},\bar{x})^{T}\mu.
\]
If $v \neq 0$ then, by taking $w = -v$, the imposed assumptions imply the existence of some $\tilde{w}$ with $\nabla_{2}\tilde{q}(\bar{p},\bar{x})\tilde{w}\in \mathcal{F}_{1}- \mathcal{F}_{2}$  fulfilling condition (\ref{eq-306}). This results  in the contradiction
\[
0 < w^{T}\nabla_{2} \Lag_\lb (\bar{p},\bar{x})\tilde{w}=\mu^{T}\nabla_{2}\tilde{q}(\bar{p},\bar{x})\tilde{w}\leq 0,
\]
where the last inequality follows from $\mu \in (\mathcal{F}_{1}- \mathcal{F}_{2})^{\circ}$ and
$\nabla_{2}\tilde{q}(\bar{p},\bar{x})\tilde{w} \in \mathcal{F}_{1}- \mathcal{F}_{2}$. Thus one has $v=0$ implying $\nabla_{2}\tilde{q}(\bar{p},\bar{x})^{T}\mu =0$ and $0 \neq p^{*} = \nabla_{1}\tilde{q}(\bar{p},\bar{x})^{T}\mu$. But from (\ref{eq-305}) we obtain
$\nabla_{1}\tilde{q}(\bar{p},\bar{x})^{T}\mu =0$  and consequently $p^{*}=0$, a contradiction. Hence \eqref{eq-301} holds true.
\endproof

The metric subregularity assumption arising in   Theorem \ref{Thm.4} can be ensured together with condition (\ref{eq-301}) in an elegant way shown in the next statement.

\begin{corollary}\label{cor.1}
Assume that for every $h \in \mathbb{R}^{m}$ there is some $k \in \mathbb{R}^{n}$ and some $\eta \in \mathbb{R}^{s}$ satisfying (\ref{eq-500}) and assume that for every nonzero $(h,k) \in \mathbb{R}^{m} \times \mathbb{R}^{n}, \eta \in \mathbb{R}^{s}$ verifying (\ref{eq-500}) and for every pair of faces $\mathcal{F}_{1}, \mathcal{F}_{2}$ of the critical cone $\mathcal{K}_{D} (\tilde{q}(\bar{p},\bar{x}),\bar{\lambda})$ with $\nabla \tilde{q}(\bar{p},\bar{x})(h,k) \in \mathcal{F}_{2} \subset \mathcal{F}_{1} \subset [\eta]^{\perp}$ there holds
\begin{equation}\label{eq-307}
 \nabla_{2}\tilde{q}(\bar{p},\bar{x})^{T}\mu = 0, ~~\mu \in (\mathcal{F}_{1}- \mathcal{F}_{2})^{\circ} \Rightarrow \mu = 0,
\end{equation}
and for every $w \neq 0$ with $b(\bar{p},\bar{x}) w \in \mathcal{F}_{1}- \mathcal{F}_{2}$ there is some $\tilde{w}$ with $\nabla_{2}\tilde{q}(\bar{p},\bar{x})\tilde{w} \in \mathcal{F}_{1}- \mathcal{F}_{2} $ and
\begin{equation}\label{eq-308}
w^{T} \nabla_{2}\Lag_\lb (\bar{p},\bar{x})\tilde{w}>0.
\end{equation}
Then all assertions of Theorem \ref{Thm.4} remain valid.
\end{corollary}

\proof
The proof easily follows from the observations that (\ref{eq-307}) implies both (\ref{eq-301}) and the metric subregularity of $F$ at $((\bar{p},\bar{x},\bar{\lambda}),0)$ in direction $(h,k,\eta)$ by virtue of Proposition \ref{LemDirSubreg}.
\endproof

We now give a simpler criterion for verifying condition \eqref{eq-307}. Consider the following lemma.
\begin{lemma}\label{LemAuxFaces}
  Let $K\subset \R^s$ be a convex polyhedral cone and let $v\in K$. Then for every pair $\F_1,\F_2$ of faces of $K$ with $v\in \F_2\subset \F_1$ there holds
  $(\F_1-\F_2)^\circ\subset\rm{sp\,}N_K(v)$.
\end{lemma}
\begin{proof}
  Since $K$ is assumed to be convex polyhedral, there are finitely many vectors $a_1,\ldots, a_t\in \R^s$ such that $K=\{u\in\R^s\mv a_i^Tu\leq 0,\ i=1,\ldots,t\}$. Consider two faces $\F_1,\F_2$ of $K$ satisfying $v\in\F_2\subset\F_1$. Then we can find index sets $I_j\subset \{1,\ldots,t\}$, $j=1,2$, such that
  \[\F_j=\{u\mv a_i^Tu=0,\ i\in I_j,\ a_i^Tu\leq 0, i\in\{1,\ldots,t\}\setminus I_j\}, j=1,2\]
  and, by a possible enlargement of  $I_2$, there exist some $\bar u\in \F_2$ with
  \[a_i^T\bar u=0,\ i\in I_2,\ a_i^T\bar u<0,\ i\in\{1,\ldots,t\}\setminus I_2.\]
  Then $I_1\subset I_2$. Indeed, assuming on the contrary that there is some $\bar i\in I_1\setminus I_2$,  we have $a_{\bar i}^T\bar u=0$ because of $\bar u\in \F_2\subset \F_1$ and $\bar i\in I_1$. On the other hand we have $a_{\bar i}^T\bar u<0$ because of $\bar i\not\in I_2$ and this is clearly impossible. Hence $I_1\subset I_2$. Further we claim that
  \[\F_1-\F_2=R:=\{u\mv a_i^Tu=0,\ i\in I_1,\ a_i^Tu\leq 0,\ i\in I_2\setminus I_1\}.\]
  The inclusion $\F_1-\F_2\subset R$ immediately follows. To prove the opposite inclusion consider $u\in R$. Then we can choose $\lambda\geq 0$ large enough such that $u_1:=u+\lambda\bar u$ fulfills $a_i^Tu_1<0$, $i\in \{1,\ldots,t\}\setminus I_2$. Together with $a_i^Tu_1=a_i^Tu$, $i\in I_2$ and $u\in R$ it follows that $u_1\in\F_1$. Since $u_2:=\lambda \bar u\in\F_2$, we have $u=u_1-u_2\in\F_1-\F_2$ showing $R\subset\F_1-\F_2$. Thus our claim holds true and we obtain
  \[(\F_1-\F_2)^\circ=\Big\{\sum_{i\in I_2}\mu_ia_i\mv \mu_i\geq 0,\ i\in I_2\setminus I_1\Big\}.\]
  On the other hand we have $N_K(v)=\{\sum_{i\in I(v)}\mu_ia_i\mv \mu_i\geq 0, i\in I(v)\}$, where $I(v):=\{i\in\{1,\ldots,t\}\mv a_i^Tv=0\}$ and thus ${\rm sp\,}N_K(v)=\{\sum_{i\in I(v)}\mu_ia_i\mv \mu_i\in\R, i\in I(v)\}$. Because of $v\in \F_2$ we have $I_2\subset I(v)$ and the asserted inclusion $(\F_1-\F_2)^\circ\subset{\rm sp\,}N_K(v)$ follows.
\end{proof}
On the basis of Corollary \ref{cor.1} and Lemma \ref{LemAuxFaces} we can now state an efficient variant of Theorem \ref{Thm.5} in which the manipulation with faces of $\K_D(\tilde q(\pb,\xb),\lb)$ is reduced only to the verification of \eqref{eq-306}.
\begin{theorem}\label{Thm.6}
Assume that $\lb$ is the (unique) solution of system (\ref{eq-61}) and for every $h \in \mathbb{R}^{m}$ there is some $k \in\mathbb{R}^{n}$ and some $\eta \in \mathbb{R}^{s}$ fulfilling \eqref{eq-500}.

Further assume that for every nonzero pair $(h,k)\in \mathbb{R}^{m} \times \mathbb{R}^{n}$ and the corresponding (unique) $\eta \in \mathbb{R}^{s}$ satisfying (\ref{eq-500})
one has
\begin{enumerate}
\item[(i)]
\begin{equation}\label{eq-950}
\nabla_2\tilde q(\pb,\xb)^T\mu=0,\ \mu\in {\rm sp\,}\big(N_{\K_D(\tilde q(\pb,\xb),\lb)}(\nabla \tilde q(\pb,\xb)(h,k))\big)\ \Rightarrow\ \mu=0;
\end{equation}
\item[(ii)]
for every pair of faces $\mathcal{F}_{1}, \mathcal{F}_{2}$  of the critical cone $\mathcal{K}_{D} (\tilde{q}(\bar{p},\bar{x}),\bar{\lambda})$ with $\nabla \tilde{q}(\bar{p},\bar{x})(h,k) \in \mathcal{F}_{2} \subset \mathcal{F}_{1} \subset [\eta]^{\perp}$
and for every $w \neq 0$ with $b(\bar{p},\bar{x})w \in \mathcal{F}_{1}- \mathcal{F}_{2} $ there is some $\tilde{w}$ with $  \nabla_{2}\tilde{q}(\bar{p},\bar{x})\tilde{w} \in \mathcal{F}_{1}- \mathcal{F}_{2}$ and
\begin{equation}\label{eq-951}
w^{T} \nabla_{2} \Lag_\lb(\bar{p},\bar{x})\tilde{w}>0.
\end{equation}
\end{enumerate}
Then all assertions of Theorem \ref{Thm.4} remain valid.
\end{theorem}
\begin{proof}
  It follows immediately from Lemma \ref{LemAuxFaces} with $K=\K_D(\tilde q(\pb,\xb),\lambda)$ and $v=\nabla \tilde q(\pb,\xb)(h,k)$.
\end{proof}

The next example illustrates the application of the preceding result.
\begin{example}
  Consider the solution map $S$ of the variational system defined via \eqref{eq-50} with
  \[f(p,x)=\left[\begin{array}{c}x_1-p\\-x_2+x_2^2\end{array}\right],\ q(p,x,y)=\left[\begin{array}{c}p-x_1+2y_1-4y_2\\-x_1+2y_1+4y_2\end{array}\right],\ D=\R^2_-\]
  at the reference point $\pb=0$, $\xb=(0,0)$. Then
  \[b(p,x)=\left[\begin{array}
    {cc}2&-4\\ 2&4
  \end{array}\right],\quad \tilde q(p,x)=\left[\begin{array}
    {c}p+x_1-4x_2\\x_1+4x_2
  \end{array}\right]\]
  and (A) is fulfilled since $b(\pb,\xb)$ has full rank. Further,
  $\lb=(0,0)$ is the unique solution of \eqref{eq-61}, $\K_D(\tilde q(\pb,\xb),\lb)=D=\R^2_{-}$ and the system \eqref{eq-500}
  reads as
  \begin{equation}\label{EqExampleDir}\left[\begin{array}{c}0\\0\end{array}\right]
  =\left[\begin{array}{c} -h+k_1\\-k_2  \end{array}\right]+\left[\begin{array}{cc} 2&2\\-4&4\end{array}\right]\eta,\ \eta\in N_{\R^2_-}\left(\left[\begin{array}
  {c}h+k_1-4k_2\\ k_1+4k_2\end{array}\right]\right).\end{equation}
  Straightforward calculations yield that for every $h\in \R$ the set $T(h):=\{(k,\eta)\in\R^2\times\R^2\mv (h,k,\eta) \mbox{ fulfills }\eqref{EqExampleDir}\}$ is not empty and
  \begin{equation}\label{EqEx_T}T(h)=\begin{cases}\Big\{\big((h,0),(0,0)\big), \big((\frac 87 h,-\frac 27h),(0,-\frac 1{14}h) \big), \big((\frac 97 h,\frac 47 h),(-\frac 17 h,0)\big)\Big\}&\mbox{if $h<0$,}\\
  \Big\{\big((-\frac 12 h,\frac 18 h), (\frac{23}{64}h,\frac{25}{64}h)\big)\Big\}&\mbox{if $h\geq 0$.}\end{cases}\end{equation}
  Thus for every $h\in\R$ there is some pair $(k,\eta)\in\R^2\times\R^2$ fulfilling \eqref{eq-500} and we shall now show  that the other assumptions of Theorem \ref{Thm.6} are fulfilled as well. Note that \eqref{eq-950} always holds because the matrix
  \[\nabla_2 \tilde q(\pb,\xb)=\left[\begin{array}{cc}1&-4\\1&4 \end{array}\right]\]
    has full rank. According to \eqref{EqEx_T} we have to consider the following four cases.\\
    \indent{\bf Case (i):}  $h>0 $, $k=(-\frac 12 h,\frac 18 h)$, $\eta= (\frac{23}{64}h,\frac{25}{64}h)$. Evidently, $\F_2=\F_1=\{0\}$ is the only face of $\K_D(\tilde q(\pb,\xb),\lb)=\R^2_{-}$ contained in $[\eta]^\perp$. Thus $w=0$ is the solely  element satisfying $b(\pb,\xb)w\in\F_1-\F_2=\{0\}$ and we are done.\\
    \indent{\bf Case (ii):} $h<0$, $k=(h,0)$, $\eta=(0,0)$. In this case we have the requirement
    \[\nabla \tilde q(\pb,\xb)(h,k)=\left[\begin{array}{c}2h\\h \end{array}\right]\in \F_2\subset \F_1\subset\R^2\]
    resulting in $\F_2=\F_1=\R^2_-$ and $\F_1-\F_2=\R^2$. Consider any $w\in\R^2\setminus \{0\}$. Then, by taking $\tilde w=(w_1,-w_2)$ we have $\nabla_2\tilde q(\pb,\xb)\tilde w\in\F_1-\F_2$ and $w^T\nabla_2\Lag_\lb(\pb,\xb)\tilde w=w_1^2+w_2^2>0$ showing the validity of \eqref{eq-951}.\\
    \indent{\bf Case (iii):} $h<0$, $k=(\frac 87 h,-\frac 27h)$, $\eta=(0,-\frac 1{14}h)$. In this case we conclude from the condition
    \[\nabla \tilde q(\pb,\xb)(h,k)=\left[\begin{array}{c}\frac {23}7 h\\0 \end{array}\right]\in \F_2\subset \F_1\subset[\eta]^\perp\]
    that $\F_2=\F_1=\R_-\times\{0\}$ and thus $\F_1-\F_2=\R\times \{0\}$. Any $w\not=0$ with $b(\pb,\xb)w\in\F_1-\F_2$ satisfies $w_2=-\frac{w_1}2\not=0$ and by choosing $\tilde w=(w_1, -\frac{w_1}4)$ we have $\nabla_2\tilde q(\pb,\xb)\tilde w\in\F_1-\F_2$ and $w^T\nabla_2\Lag_\lb(\pb,\xb)\tilde w=w_1^2-w_2\tilde w_2=\frac 78 w_1^2>0$ verifying again \eqref{eq-951}.\\
    \indent{\bf Case (iv):} $h<0$, $k=(\frac 97 h,\frac 47 h)$, $\eta=(-\frac 17 h,0)$. In this case the faces $\F_1,\F_2$ satisfying
    \[\nabla \tilde q(\pb,\xb)(h,k)=\left[\begin{array}{c}0\\\frac{25}7 h \end{array}\right]\in \F_2\subset \F_1\subset[\eta]^\perp\]
    are $\F_1=\F_2=\{0\}\times\R_-$ and thus $\F_1-\F_2=\{0\}\times\R$. Any $w\not=0$ with $b(\pb,\xb)w\in\F_1-\F_2$ satisfies $w_2=\frac{w_1}2\not=0$ and  $\tilde w=(w_1, \frac{w_1}4)$ fulfills $\nabla_2\tilde q(\pb,\xb)\tilde w\in\F_1-\F_2$ and $w^T\nabla_2\Lag_\lb(\pb,\xb)\tilde w=w_1^2-w_2\tilde w_2=\frac 78 w_1^2>0$. Hence, \eqref{eq-951} holds in this case as well.\\
    Thus all assumptions of Theorem \ref{Thm.6} are fulfilled and the solution map $S$ has the Aubin property around $(\pb,\xb)$.

   Note that this result cannot be obtained by condition (5.2) from \cite{MO} which attains the form
    \[
    \begin{split}
     p^{*} & = -   v_{1} + w_{1}\\
     0 &  =  v_{1} + w_{1} + w_{2}\\
     0 &  = -  v_{2} + 4 w_{1}+ 4 w_{2}
    \end{split} w \in D^{*}N_{\mathbb{R}^{2}_{-}}(0,0)
    \left(\left[\begin{split}
    & 2 v_{1}- 4 v_{2}\\
    & 2 v_{1}+ 4 v_{2}
    \end{split}\right]\right) \Rightarrow
    p^{*}= v_{1} = v_{2} = 0.
    \]
    Indeed, the relations on the left-hand side have, e.g., the nontrivial solution $v_{1}=-1,\ v_{2}=-0.5,\ p^{*}= \frac{25}{16}$.
\end{example}


\section*{Conclusion}
The paper contains a thorough analysis of a parameterized variational system with implicit constraints. One can say that Boris Mordukhovich stands behind most important ingredients used in  this development. Indeed, as pointed out in the Introduction, the model came from \cite{QVI} and the results in Section 4 are in fact directional variants of their counterparts in \cite[Section 3]{QVI}. Furthermore, the development of the directional limiting calculus has been initiated in \cite{GiMo} and also Theorem \ref{Thm.1} \cite[Theorem 4.4]{GO3} relies essentially on the so-called Mordukhovich criterion \cite[Chapter 9F]{RoWe98}. Thus, via this research  the authors would like to give credit to their friend Boris on the occasion of his 70th birthday.

\section*{Acknowledgements} The research of the first author was  supported by the Austrian Science Fund (FWF) under grant P29190-N32. The  research of the second author was supported by the Grant Agency of the Czech Republic, Projects 17-08182S and 17-04301S and the Australian Research Council, Project DP160100854.

\end{document}